\documentclass[12 pt]{article}
\usepackage[pdftex]{graphicx}
\usepackage{amsmath}
\usepackage{mathtools}
\usepackage{amsthm}
\usepackage{amsfonts}
\usepackage{amssymb}
\usepackage[margin=1.0in]{geometry}
\usepackage{tikz-cd}
\usetikzlibrary{cd}
\usepackage{enumitem}
\usepackage{url}

\newtheorem{theorem}{Theorem}

\newtheorem{lemma}[theorem]{Lemma}
\newtheorem{prop}[theorem]{Proposition}

\theoremstyle{definition}
\newtheorem{defin}[theorem]{Definition}
\newtheorem{fact}[theorem]{Fact}

\newtheorem{que}[theorem]{Question}

\theoremstyle{remark}
\newtheorem*{rem}{Remark}
\newtheorem*{claim}{Claim}

\newcommand{\fin}[1]{\mathrm{Fin}(#1)}

\renewcommand{\phi}{\varphi}

\newcommand{\aut}[1]{\mathrm{Aut}(#1)}

\newcommand{\im}[1]{\mathrm{Im}(#1)}

\begin{document}
	\title{Minimal flows with arbitrary centralizer}
	\author{Andy Zucker}
	\date{September 2019}
	\maketitle
	
\begin{abstract}
	Given any pair of countable groups $G$ and $H$ with $G$ infinite, we construct a minimal, free, Cantor $G$-flow $X$ so that $H$ embeds into the group of automorphisms of $X$. This generalizes results of \cite{CP} and \cite{GTWZ}.
	\let\thefootnote\relax
	\let\thefootnote\relax\footnote{The author was supported by NSF Grant no.\ DMS 1803489.}
\end{abstract}
	
Let $G$ be an infinite countable group, and let $X$ be a $G$-flow, i.e.\ a compact Hausdorff space equipped with a continuous $G$-action. A $G$-flow is called \emph{minimal} if every orbit is dense. The \emph{centralizer} of $X$ is the group of all homeomorphisms of $X$ which commute with the $G$-action. We denote this group by $\aut{X}$ when the action is understood.

The study of the centralizers of $G$-flows has been an active area of research, especially in the case $G = \mathbb{Z}$. Usually some constraint is placed upon the flows $X$ under consideration, for instance by demanding that $X$ is a subshift over a finite alphabet (see for instance \cite{BLR}, \cite{CK}, \cite{DDMP}, \cite{FST}, and \cite{Hed}). More recently, interest has turned to just considering minimality with no other constraints. Namely, does only the knowledge that $X$ is a minimal $G$-flow place any algebraic constraints on the possible groups that can appear as $\aut{X}$? 

A natural constraint to place on $X$ is that the underlying space of $X$ be the Cantor space. We call $G$-flows with this property \emph{Cantor} flows. This is not much of a constraint at all, since every countable group can act freely on Cantor space \cite{HM}. Cortez and Petite in \cite{CP} construct for every residually finite countable group $H$ a minimal Cantor $\mathbb{Z}$-flow $X$ such that $H$ embeds into $\aut{X}$. Independently and using different techniques, Glasner, Tsankov, Weiss and Zucker in \cite{GTWZ} construct for any countable group $G$ and any countable group $H$ which embeds into a compact group a free, minimal, Cantor $G$-flow $X$ for which $H$ embeds into $\aut{X}$. Recall that the $G$-flow $X$ is \emph{free} if for every $g\in G\setminus \{1_G\}$ and $x\in X$, we have $gx\neq x$.

The goal of this paper is to prove the following theorem.

\begin{theorem}
	\label{Thm:Main}
	Let $G$ and $H$ be any countable groups with $G$ infinite. Then there is a minimal, free, Cantor $G$-flow $X$ so that $H$ embeds into $\aut{X}$.
\end{theorem} 	

We may assume without loss of generality that $H$ is also infinite. We also note that to prove Theorem~\ref{Thm:Main}, it suffices to construct any minimal $G$-flow $X$ so that $H$ embeds into $\aut{X}$. If $X$ is a minimal $G$-flow such that $H$ embeds into $\aut{X}$, then by Theorem 1.2 of \cite{GTWZ}, there is a minimal, free $G$-flow $Y$ with $X\times Y$ also minimal. Then by arguing as in Theorem 11.5 of \cite{GTWZ}, we can find $Z$ a suitable highly proximal extension of $X\times Y$ which is homeomorphic to Cantor space and such that $H$ still embeds into $\aut{Z}$. However, it seems very likely that the construction given here gives an \emph{essentially free} $G$-flow, in which case the appeal to Theorem 1.2 of \cite{GTWZ} is not needed.

We start with two preliminary sections. The first is on \emph{blueprints}, a notion developed by Gao, Jackson, and Seward in \cite{GJS}. The second discusses strongly irreducible subshifts. The final section proves Theorem~\ref{Thm:Main}.
\vspace{3 mm}

\noindent
\emph{Acknowledgements} I would like to thank Eli Glasner for pointing out to me a mistake in an earlier version of \cite{CP}, which directly inspired the work here.
\section{Blueprints}

The notion of a \emph{blueprint} is developed by Gao, Jackson, and Seward in \cite{GJS}, where in particular, it is proven that every group carries a non-trivial blueprint. To keep this paper self-contained, we provide a proof of this. We delay the definition of a blueprint until we have actually constructed one. Throughout this section, we will use the group $G$; the group $H$ will figure more heavily in the next section.

For this section, we fix an exhaustion $G = \bigcup_n A_n$, where each $A_n$ is finite, symmetric, and contains the identity $1_G\in G$. We denote this exhaustion by $\mathcal{A}$. We assume that each $A_n$ is large enough to write $A_n = A_0^3\cdot A_1^3\cdots A_{n-1}^3\cdot B_n$ for some finite set $B_n$ containing $1_G$ which we now fix. In particular, notice that $A_0^3\cdots A_{n-1}^3\subseteq A_n$. Given $k< n$, we set $A_n(k) = A_k^3\cdots A_{n-1}^3\cdot B_n$. Notice that if $k'\leq k$, then $A_n(k)\subseteq A_n(k')$. Also notice that $A_kA_n(k)\supseteq A_n$.

If $F\subseteq G$ is finite, we say that $g, h\in G$ are \emph{$F$-apart} if $Fg\cap Fh = \emptyset$. We say that $S\subseteq G$ is \emph{$F$-spaced} if every $g\neq h\in S$ is $F$-apart.
\vspace{2 mm}

\begin{defin}
	\label{Def:System}
	An \emph{$\mathcal{A}$-system of height $n$} is a collection $\mathcal{S} = \{\mathcal{S}(0),...,\mathcal{S}(n)\}$ of subsets of $A_n$ defined by reverse induction as follows.
	\begin{itemize}
		\item 
		$\mathcal{S}(n) = \{1_G\}$
		\item 
		If $\mathcal{S}(k+1),...,\mathcal{S}(n)$ have all been defined, we say that $g\in A_n$ is \emph{$k$-admissible for $\mathcal{S}$} if, letting $\ell > k$ be least with $A_k\cdot g\cap A_\ell\cdot \mathcal{S}(\ell) \neq \emptyset$, then there is $h\in \mathcal{S}(\ell)$ with $A_k\cdot g\subseteq A_\ell(k)\cdot h$. Write $\mathrm{Ad}(k, \mathcal{S})$ for the set of $g\in A_n$ which are $k$-admissible for $\mathcal{S}$.
		\item 
		$\mathcal{S}(k)$ is any maximal $A_k$-spaced subset of $\mathrm{Ad}(k, \mathcal{S})$ containing $1_G$.
	\end{itemize}
For the last item, notice by reverse induction that $1_G\in \mathrm{Ad}(k, \mathcal{S})$ for each $k< n$.
\end{defin}	
\vspace{2 mm}

\noindent
Let us immediately clarify an important point about the set $\mathrm{Ad}(k, \mathcal{S})$.
\vspace{2 mm}

\begin{lemma}
	\label{Lem:GeneralAd}
	Suppose $g\in \mathrm{Ad}(k, \mathcal{S})$. Then for any $m> k$ with $A_k\cdot g\cap A_m\cdot \mathcal{S}(m) = \emptyset$, then there is $b\in \mathcal{S}(m)$ with $A_k\cdot g\subseteq A_m(k)\cdot b$.
\end{lemma}

\begin{proof}
	We induct on $m-k$. When $m-k = 1$, the lemma follows from the definitions. If $m-k > 1$, then consider the least $\ell > k$ with $A_k\cdot g \cap A_\ell \mathcal{S}(\ell) \neq \emptyset$. Using item (2) of the definition, there is $h\in \mathcal{S}(\ell)$ with $A_k\cdot g\subseteq A_\ell(k)\cdot h$. If $\ell = m$ we are done. If $\ell < m$, then $A_\ell\cdot h\cap A_m\cdot \mathcal{S}(m)\neq \emptyset$, so by induction we can find $b\in \mathcal{S}(m)$ with $A_\ell\cdot h\subseteq A_m(\ell)\cdot b$. Then $A_k\cdot g\subseteq A_\ell\cdot h\subseteq A_m(\ell)\cdot b\subseteq A_m(k)\cdot b$.
\end{proof}
\vspace{2 mm}

For the moment, fix an $\mathcal{A}$-system $\mathcal{S}$ of height $n$. Our first main goal is Proposition~\ref{Prop:SystemsSyndetic}, which shows that the sets $\mathcal{S}(k)$ are somewhat large.
\vspace{2 mm}

\begin{lemma}
	\label{Lem:AdmissibleSyndetic}
	Suppose $g\in \mathrm{Ad}(k, \mathcal{S})$. Then $A_k^2\cdot g\cap \mathcal{S}(k)\neq \emptyset$.
\end{lemma}	

\begin{proof}
	If $A_k^2\cdot g\cap \mathcal{S}(k)$ were empty, then $\mathcal{S}(k)\cup \{g\}$ would be a strictly larger $A_k$-spaced subset of $\mathrm{Ad}(k, \mathcal{S})$.
\end{proof}
\vspace{2 mm}

\begin{lemma}
	\label{Lem:InsideAd}
	Suppose $\ell > k$ and $h\in \mathcal{S}(\ell)$. Then $A_k^2\cdot A_\ell(k+1)\cdot h\setminus A_k\cdot A_\ell(k+1)\cdot h\subseteq \mathrm{Ad}(k, \mathcal{S})$.
\end{lemma}

\begin{proof}
	Fix $g$ in the left hand side. Then $A_k\cdot g\subseteq A_\ell(k)\cdot h\setminus A_\ell(k+1)\cdot h$. Towards a contradiction, suppose there were some $m$, $k< m < \ell$, with $A_k\cdot g\cap A_m\cdot \mathcal{S}(m) \neq \emptyset$. Suppose $b\in \mathcal{S}(m)$ satisfies $A_k\cdot g\cap A_m\cdot b \neq \emptyset$. Then since $b\in \mathrm{Ad}(m, \mathcal{S})$, Lemma~\ref{Lem:GeneralAd} implies that $A_m\cdot b\subseteq A_\ell(m)\cdot h$. But since we have $A_k\cdot g\cap A_\ell(k+1)\cdot h = \emptyset$, this is a contradiction.
\end{proof}
\vspace{2 mm}

\begin{defin}
	\label{Def:SyndeticIn}
	Suppose $F\subseteq G$ is finite, $D\subseteq G$, and let $S\subseteq D$. We say that $S$ is \emph{$F$-syndetic in $D$} if for any $g\in G$ such that $Fg\subseteq D$, we have $Fg\cap S\neq \emptyset$. If $D = G$, we simply say that $S$ is \emph{$F$-syndetic}. We say that $S$ is \emph{syndetic} if there is a finite $F\subseteq G$ so that $S$ is $F$-syndetic.
\end{defin}	
\vspace{2 mm}

\begin{prop}
	\label{Prop:SystemsSyndetic}
	The set $\mathcal{S}(k)\subseteq A_n$ is $A_k^5$-syndetic in $A_n$.
\end{prop}

\begin{proof}
	Suppose we have $g\in G$ with $A_k^5\cdot g\subseteq A_n$. If $g\in \mathrm{Ad}(k, \mathcal{S})$, we are done by Lemma~\ref{Lem:AdmissibleSyndetic}, so assume this is not the case. Let $\ell > k$ be least with $A_k\cdot g\cap A_\ell\cdot \mathcal{S}(\ell) \neq \emptyset$, and fix some $h\in \mathcal{S}(\ell)$ and $f\in A_k$ with $fg\in A_\ell\cdot h$. Notice that we cannot have $fg\in A_k\cdot A_\ell(k+1)\cdot h$, as this would imply that $g\in \mathrm{Ad}(k, \mathcal{S})$. In particular, for some $i\in \{2, 3, 4\}$, we have $fg\in A_k^i\cdot A_\ell(k+1)\cdot h\setminus A_k^{i-1}\cdot A_\ell(k+1)\cdot h$. In each case, we can find $f_0\in A_k^2$ with $f_0fg\in A_k^2\cdot A_\ell(k+1)\cdot h\setminus A_k\cdot A_\ell(k+1)\cdot h$. By Lemma~\ref{Lem:InsideAd}, we have $f_0fg\in \mathrm{Ad}(k, \mathcal{S})$, so by Lemma~\ref{Lem:AdmissibleSyndetic}, we have $A_k^2\cdot f_0fg\cap \mathcal{S}(k)\neq \emptyset$. We are done once we note that $A_k^2\cdot f_0f\subseteq A_k^5$.
\end{proof}
\vspace{2 mm}

We now investigate how to modify $\mathcal{A}$-systems to create new ones. Definition~\ref{Def:RelativeSystem} and Proposition~\ref{Prop:RelativeSystem} give a method to restrict to a smaller system, while Definition~\ref{Def:ReplaceSystem} and Proposition~\ref{Prop:ReplaceSystem} allow us to print a smaller system inside a larger one.
\vspace{2 mm}

\begin{defin}
	\label{Def:RelativeSystem}
	Suppose $g\in \mathcal{S}(m)$. Then $(g\cdot \mathcal{S})|_m = \{(g\cdot \mathcal{S})|_m(0),..., (g\cdot \mathcal{S})|_m(m)\}$ denotes the $\mathcal{A}$-system of height $m$ where for $k\leq m$ we set $(g\cdot \mathcal{S})|_m(k) = (\mathcal{S}(k)\cap A_m\cdot g)\cdot g^{-1}$. If $g = 1_G$, we simply write $\mathcal{S}|_m$.
\end{defin}
\vspace{2 mm}

\begin{prop}
	\label{Prop:RelativeSystem}
	 $(g\cdot \mathcal{S})|_m$ is an $\mathcal{A}$-system of height $m$.
\end{prop}

\begin{proof}
	We proceed by reverse induction on $k< m$. First we note that $\mathrm{Ad}(k, (g\cdot\mathcal{S})|_m) = (\mathrm{Ad}(k, \mathcal{S})\cap A_m\cdot g)\cdot g^{-1}$. Then, if $b, h\in \mathrm{Ad}(k, \mathcal{S})$ with $A_k\cdot b\cap A_m\cdot g = \emptyset$ and $A_k\cdot h\subseteq A_m(k)\cdot g$, then we have $A_k\cdot b\cap A_k\cdot h = \emptyset$. It follows that $(\mathcal{S}(k)\cap A_k\cdot g)\cdot g^{-1}$ is a maximal $A_k$-spaced subset of $\mathrm{Ad}(k, (g\cdot \mathcal{S})|_m)$. 
\end{proof}
\vspace{2 mm}

\begin{defin}
	\label{Def:ReplaceSystem}
	Let $\mathcal{S}$ be an $\mathcal{A}$-system of height $n$. Let $\mathcal{T}$ be an $\mathcal{A}$-system of height $m$ for some $m< n$. Given $g\in \mathcal{S}(m)$, we let $(\mathcal{S}, \mathcal{T}, g)$ denote the $\mathcal{A}$-system of height $n$ where for $k\leq n$, we have 
	\begin{itemize}
		\item 
		$(\mathcal{S}, \mathcal{T}, g)(k) = \mathcal{S}(k)$ for $m\leq k \leq n$.
		\item 
		$(\mathcal{S}, \mathcal{T}, g)(k) = (\mathcal{S}(k)\setminus A_m\cdot g)\cup \mathcal{T}(m)\cdot g$ for $k < m$.
	\end{itemize}
\end{defin}
\vspace{2 mm}

\begin{prop}
	\label{Prop:ReplaceSystem}
	$(\mathcal{S}, \mathcal{T}, g)$ is an $\mathcal{A}$-system of height $n$. 
\end{prop}

\begin{proof}
	We proceed by reverse induction on $k \leq n$. For $k\geq m$ there is nothing to prove. For $k < m$, we observe that $\mathrm{Ad}(k, (\mathcal{S}, \mathcal{T}, g)) = (\mathrm{Ad}(k, \mathcal{S})\setminus A_m\cdot g)\cup \mathrm{Ad}(k, \mathcal{T})\cdot g$. Then we note that $\mathcal{S}(k)\setminus A_m\cdot g$ and $\mathcal{T}(m)\cdot g$ are $A_k$-apart. It follows that $(\mathcal{S}(k)\setminus A_m\cdot g)\cup \mathcal{T}(m)\cdot g$ is a maximal $A_k$-spaced subset of $\mathrm{Ad}(k, (\mathcal{S}, \mathcal{T}, g))$.
\end{proof}
\vspace{2 mm}

We use Proposition~\ref{Prop:ReplaceSystem} to construct particularly nice $\mathcal{A}$-systems.
\vspace{2 mm}

\begin{defin}
	\label{Def:UniformSystem}
	Let $\mathcal{S}$ be an $\mathcal{A}$-system of height $n$. We call $\mathcal{S}$ \emph{uniform} if $(g\cdot \mathcal{S})|_m = (h\cdot \mathcal{S})|_m$ for any $g, h\in \mathcal{S}(m)$ and any $m\leq n$.
\end{defin}
\vspace{2 mm}

\begin{prop}
	\label{Prop:ExistsUniform}
	There is a sequence $\{\mathcal{S}_n: n< \omega\}$ of uniform $\mathcal{A}$-systems such that $\mathcal{S}_n$ has height $n$ and $\mathcal{S}_n|_m = \mathcal{S}_m$ for any $m\leq n$.
\end{prop}

\begin{proof}
	We proceed by (forward) induction. For $n = 0$ the unique $\mathcal{A}$-system of height zero is vacuously uniform. Suppose $\mathcal{S}_0,...,\mathcal{S}_{n-1}$ have been constructed. Let $\mathcal{T} := \mathcal{T}_0$ be any $\mathcal{A}$-system of height $n$. For each $k < n$, we set
	$$T(k) = \mathcal{T}(k)\setminus \left(\bigcup_{k< m< n} A_m\cdot \mathcal{T}(m)\right)$$
	Note that the sets $T(0),...,T(n-1)$ are pairwise disjoint. Fix some enumeration of $\bigcup_{k< n} T(k) = \{g_0,...,g_{r-1}\}$, and for each $i< r$, let $\phi(i)< n$ be the unique index with $g_i\in T(\phi(i))$. We repeatedly use Proposition~\ref{Prop:ReplaceSystem} to define $\mathcal{A}$-systems $\mathcal{T}_0,...,\mathcal{T}_r$, and we set $\mathcal{S}_n = \mathcal{T}_r$. If $\mathcal{T}_i$ has been built for some $i < r$, we set $\mathcal{T}_{i+1} = (\mathcal{T}_i, \mathcal{S}_{\phi(i)}, g_i)$. Then $\mathcal{S}_n$ is a uniform $\mathcal{A}$-system of height $n$ as desired.
\end{proof}
\vspace{2 mm}

\begin{defin}\mbox{}
	\label{Def:Blueprint}
	\vspace{-3 mm}
	
	\begin{enumerate}
		\item 
		A sequence $\vec{\mathcal{S}}:= \{\mathcal{S}_n: n< \omega\}$ constructed as in Proposition~\ref{Prop:ExistsUniform} will be called a \emph{coherent sequence}.
		\item 
		Let $\vec{\mathcal{S}}$ be a coherent sequence. The \emph{blueprint} of $\vec{\mathcal{S}}$ is the sequence $\{\vec{\mathcal{S}}(n): n< \omega\}$, where $\vec{\mathcal{S}}(n) = \bigcup_{N\geq n} \mathcal{S}_N(n)$. We note the following properties of the blueprint of $\vec{\mathcal{S}}$:
		\begin{enumerate}
			\item 
			$\vec{\mathcal{S}}(n)\supseteq \vec{\mathcal{S}}(n+1)$, and each $\vec{\mathcal{S}}(n)$ is $A_n$-spaced and $A_n^5$-syndetic.
			\item 
			For any $k\leq n$, $g\in \vec{\mathcal{S}}(k)$, and $h\in \vec{\mathcal{S}}(n)$, we either have $A_k\cdot g\cap A_n
			\cdot h = \emptyset$ or $A_k\cdot g\subseteq A_n(k)\cdot h$.
			\item 
			For any $k\leq n$ and $g, h\in \vec{\mathcal{S}}(n)$, we have $(\vec{\mathcal{S}}(k)\cap A_n\cdot g)g^{-1} = (\vec{\mathcal{S}}(k)\cap A_n\cdot h)h^{-1}$.
			\item 
			For each $n< \omega$, we have $|\vec{\mathcal{S}}(n)\cap A_{n+1}|\geq |A_n^2\cdot B_{n+1}|/|A_n^5|$.
		\end{enumerate} 
	\end{enumerate}
\end{defin}

\begin{rem}
	Compare this to Definition 5.1.2 of \cite{GJS}. In fact, we have constructed what they call a \emph{centered} blueprint.
\end{rem}
\vspace{2 mm}

\section{Strongly irreducible subshifts}

In this section, we work with the group $H$. If $M$ is a compact space, then $H$ acts on the space $M^H$ by shift, where given $x\in M^h$ and $g, h\in H$, we set $g\cdot x(h) = x(hg)$. A \emph{subshift} is any non-empty closed $X\subseteq M^H$ which is $H$-invariant. Most of the time, $X$ will be a finite set $A$. Let $X\subseteq A^H$ be a subshift. If $C\subseteq H$ is finite, the set of \emph{$C$-patterns} of $X$ is given by $P_C(X) = \{x|_C: x\in X\}\subseteq A^C$. If $D\subseteq H$ is finite, recall the definitions of $D$-spaced and $D$-apart given immediately before Definition~\ref{Def:System}. 
\vspace{2 mm}

\begin{defin}
	\label{Def:StrIrred}
	Let $D\subseteq H$ be finite. A subshift $X\subseteq A^H$ is \emph{$D$-irreducible} if for any $S_0, S_1\subseteq H$ which are $D$-apart and any $x_0, x_1\in X$, there is $y\in X$ such that $y|_{S_i} = x_i|_{S_i}$ for each $i < 2$. We sometimes say that $y$ \emph{blends} $x_0|_{S_0}$ and $x_1|_{S_1}$. We say that $X$ is \emph{strongly irreducible} if $X$ is $D$-irreducible for some finite $D\subseteq H$.
\end{defin}
\vspace{2 mm}

\begin{fact}
		Let $A$ and $B$ be finite sets. If $X\subseteq A^H$ is $D_X$-irreducible and $Y\subseteq B^H$ is $D_Y$-irreducible, then $X\times Y\subseteq (A\times B)^H$ is $(D_X\cup D_Y)$-irreducible.
\end{fact}
\vspace{2 mm}

The remainder of this section discusses some examples of strongly irreducible flows that we will use in the construction of the next section. 

Let $C\subseteq H$ be finite, and let $n = |C^{-1}C|$. By a greedy argument, there is a partition of $H$ into $n$-many $C$-spaced sets. Even better, given any $S\subseteq H$ and $\delta\colon S\to n$ such that $\delta^{-1}(k)$ is $C$-spaced for each $k< n$, we can extend $\delta$ to some $\gamma\colon H\to n$ such that $\gamma^{-1}(k)$ is $C$-spaced for each $k< n$. We set
$$\mathrm{Part}(C, n) := \{\gamma\in n^H: \gamma^{-1}(k)\text{ is $C$-spaced for each $k< n$}\}$$
and note that $\mathrm{Part}(C, n)$ is $C$-irreducible.

Now let $C, D\subseteq H$ be finite, and suppose $X\subseteq A^H$ is $D$-irreducible, and fix $\alpha\in P_C(X)$. Suppose $S\subseteq H$ is $DC$-spaced. Then by repeatedly using $D$-irreducibility, we can find $x\in X$ such that $hx|_C = \alpha$ for each $h\in S$. Letting $N = |(DC)^{-1}DC|$, we then set
\begin{align*}
\mathrm{Print}(X, \alpha, N) :=& \{(x_0,...,x_{N-1})\in X^N: \exists \gamma\in \mathrm{Part}(DC, N)\,\, \forall h\in H\,\, hx_{\gamma(h)}|_C = \alpha\}\\[1 mm]
\subseteq& \{(x_0,...,x_{N-1})\in X^N: \forall h\in H\,\, \exists i< N\,\, hx_i|_C = \alpha\}
\end{align*}

\begin{prop}
	\label{Prop:PrintStrIrred}
	$\mathrm{Print}(X, \alpha, N)$ is $DC(DC)^{-1}D$-irreducible.
\end{prop}

\begin{proof}
	Let $(x_0,...,x_{N-1}), (y_0,...,y_{N-1})\in \mathrm{Print}(X, \alpha, N)$ as witnessed by $\gamma_x, \gamma_y\in \mathrm{Part}(DC, N)$. Let $S_x, S_y\subseteq H$ be $DC(DC)^{-1}D$-apart. For each $i < N$, we enlarge $S_x$ to a set $S_x(i)\subseteq C(DC)^{-1}DS_x$ by adding in $Ch$ if we have both:
	\begin{enumerate}
		\item 
		$\gamma_x(h) = i$,
		\item 
		$Ch$ and $S_x$ are not $D$-apart, i.e.\ if $h\in (DC)^{-1}DS_x$.
	\end{enumerate}
	Notice that for each $h\in (DC)^{-1}DS_x$, we add $Ch$ to $S_x(\gamma_x(h))$. Do the same thing for $y$.
	
	Since $\mathrm{Part}(DC, N)$ is $DC$-irreducible, we can find $\gamma\in \mathrm{Part}(DC, N)$ blending $\gamma_x|_{(DC)^{-1}DS_x}$ and  $\gamma_y|_{(DC)^{-1}DS_y}$. Notice that if $h\not\in (DC)^{-1}D(S_x\cup S_y)$ and $\gamma(h) = i$, then $Ch$ and $S_x(i)\cup S_y(i)$ are $D$-apart; this is because in forming $S_x(i)$ and $S_y(i)$, we only added sets which were $D$-apart from $Ch$. Now for each $i< N$, find $z_i\in X$ which blends $x_i|_{S_x(i)}$ and $y_i|_{S_y(i)}$ and with $hz_i|_C = \alpha$ whenever $\gamma(h) = i$. Then $(z_0,...,z_{N-1})\in \mathrm{Print}(X, \alpha, N)$ is as desired.
\end{proof}

\section{The construction}

In this section, we construct a $(G\times H)$-subshift $X\subseteq 2^{G\times H}$ which is essentially free (in fact free) as an $H$-flow and minimal as a $G$-flow. This will prove Theorem~\ref{Thm:Main}. We will often think of $2^{G\times H}$ as either the $G$-flow $(2^H)^G$ or as the $H$-flow $(2^G)^H$ as needed. We will first construct an $H$-flow $Y = \varprojlim Y_n \subseteq 2^{G\times H}$. Then we will set $X = \overline{G\cdot Y}$. The main work in this section is the construction of $Y$. 

We start by fixing both an exhaustion $G = \bigcup_n A_n$ as in section 1 and a coherent sequence $\vec{\mathcal{S}}$ on $G$. We will adhere to the notation developed in section 1 as much as possible. We will often assume that each $A_{n+1}$ is suitably large compared to $A_n$ to proceed as we need, especially in regards to item 2(d) of Definition~\ref{Def:Blueprint}. For each $n< \omega$, the $H$-flow $Y_n$ will be a subshift of $(2^{A_n})^H$, and for $m<n$, the projection $\pi^n_m\colon Y_n\to Y_m$ will be the one induced by the restriction map from $2^{A_n}$ to $2^{A_m}$. We also fix an exhaustion $H = \bigcup_n C_n$ with each $C_n$ finite, symmetric, and containing the identity $1_H\in H$.

One helpful definition will be the following. 
\vspace{2 mm}

\begin{defin}\mbox{}
	\vspace{-3 mm}
	
	\begin{enumerate}
		\item 
		Suppose $A\subseteq G$ and $\alpha\in 2^A$. Given $g\in G$, we let $g\cdot \alpha\in 2^{Ag^{-1}}$ be defined by $g\cdot \alpha(ag^{-1}) = \alpha(a)$ for $a\in A$. Note that $(g_0g_1)\cdot \alpha = g_0\cdot (g_1\cdot \alpha)$.
		\item 
		Suppose $A\subseteq G$ is finite and $z\in (2^A)^H$. Then for any $g\in G$, we define $g\cdot z\in (2^{Ag})^H$ where for $z\in Z$ and $h\in H$, we have $(gz)(h) = g\cdot (z(h))$. Again, note that $(g_0g_1)\cdot z = g_0\cdot (g_1\cdot z)$.
		\item 
		Note that if $Z\subseteq (2^A)^H$ is a subshift, then $g\cdot Z\subseteq (2^{Ag^{-1}})^H$ is also a subshift.
	\end{enumerate}
\end{defin}
\vspace{2 mm}

We build the flows $Y_n$ by induction, and we set $Y_0 = (2^{A_0})^H$. Trivially, $Y_0$ is $\{1_H\}:= D_0$-irreducible. Suppose $Y_0,...,Y_{n-1}$ have been constructed and are all $D_{n-1}$-irreducible for some finite symmetric $D_{n-1}\subseteq H$. For each $k< n$, set 
$$S_n(k) = \mathcal{S}_n(k)\setminus \left(\bigcup_{k< m< n} A_m\cdot \mathcal{S}_n(m)\right)$$
For $k< n$, we set $T_n(k) = A_k\cdot S_n(k)$. We also set $T_n(n) := A_n\setminus \bigcup_{k< n} T_n(k)$. To define $Y_n$, we will first define a subshift $Z_n\subseteq (2^{T_n(n-1)})^H$. We will then put
$$Y_n := (2^{T_n(n)})^H\times Z_n\times \prod_{k = 0}^{n-2} \prod_{g\in S_n(k)} g^{-1}Y_k$$.
We note that $Y_n$ will be strongly irreducible as long as $Z_n$ is. 

We set $|S_n(n-1)| := r$. How large does $r$ need to be? Consider the set $P_{C_{n-1}}(Y_{n-1}) := \{\alpha_0,...,\alpha_{\ell-1}\}\subseteq (2^{A_{n-1}})^{C_{n-1}}$. We will want to ensure that 
$$r > |(D_{n-1}C_{n-1})^{-1}(D_{n-1}C_{n-1})|\cdot 2^{|A_{n-1}\times C_{n-1}|}.$$ 
The size of $r$ allows us to find disjoint sets $F_i\subseteq S_n(n-1)$ for each $i< \ell$ with $|F_i| = |(D_{n-1}C_{n-1})^{-1}(D_{n-1}C_{n-1})|:= q$, while ensuring that $F_\ell := S_n(n-1)\setminus \bigcup_{i < \ell} F_i \neq \emptyset$. We also demand that $1_G\in F_\ell$. For $i< \ell$, write $F_i = \{g_0^i,...,g_{q-1}^i\}$.

Recall the flow $\mathrm{Print}$ from the previous section. We define a map
$$\Phi_i\colon \mathrm{Print}((2^{A_{n-1}})^H, \alpha_i, q)\to (2^{A_{n-1}\cdot F_i})^H$$
via $\Phi_i((x_0,...,x_{q-1})) = ((g_0^i)^{-1}x_0,...,(g_{q-1}^i)^{-1}x_{q-1})$. Note that $\Phi_i$ is injective. We set $Q_i = \im{\Phi_i}$.

We then set 
$$Z_n = \prod_{g\in F_\ell}g^{-1}Y_{n-1}\times \prod_{i< \ell} Q_i.$$ Note that $Z_n$ is strongly irreducible; hence $Y_n$ is as well. Notice that since $1_G\in F_\ell$, we have that $Y_n|_{A_{n-1}\times H} = Y_{n-1}$.
\vspace{3 mm}

This concludes the construction of $Y$. We now set $X = \overline{G\cdot Y}$.
\vspace{2 mm}

\begin{prop}
	\label{Prop:ConstructionMinimal}
	$X$ is essentially free as an $H$-flow and minimal as a $G$-flow.
\end{prop}

\begin{rem}
	Note that this immediately imples that $X$ is in fact free as an $H$-flow, since each $h\in H$ acts as an automorphism of the minimal $G$-flow $X$.
\end{rem}

\begin{proof}
	We note that each $Y_n$ is essentially free, since $Y_n|_{A_0\times H} = Y_0 = (2^{A_0})^H$. Hence $Y$ is essentially free, from which it follows that $\overline{G\cdot Y}$ is essentially free as an $H$-flow. 
	
	To show that $X$ is $G$-minimal, it suffices to show for any $x, y\in Y$ and any open $V\ni y$ that the \emph{visiting set} $\mathrm{Vis}(x, V):= \{g\in G: gx\in V\}$ is syndetic. We may assume that $V = \{z\in Y: z|_{A_{n-1}\times C_{n-1}} = y|_{A_{n-1}\times C_{n-1}} = \alpha_i\}$. Pick any $g\in \vec{\mathcal{S}}(n)$. Then $g\cdot x|_{A_n\times H} \in Y_n$. It follows that $g\cdot x|_{A_{n-1}\cdot F_i\times H}\in Q_i$. By considering $1_H\in H$ in the definition of $\mathrm{Print}$, there is $j< q$ with $g_j^i\cdot g\cdot x|_{A_{n-1}\times C_{n-1}} = \alpha_i$. It follows that $g_j^i\cdot g\in \mathrm{Vis}(x, V)$. Since $g$ was an arbitrary element of $\vec{\mathcal{S}}(n)$, an $A_n^5$-syndetic set, and since $F_i\subseteq \mathcal{S}_n(n-1)\subseteq A_n$, we see that $\mathrm{Vis}(x, V)$ is $A_n^7$-syndetic as desired. 
\end{proof}
\vspace{2 mm}

One drawback of the techniques used in this paper is the asymmetry between the roles of $G$ and $H$. For example, the following ``symmetric'' version of the result remains open.

\begin{que}
	Let $G$ and $H$ be countable infinite groups. Is there a free $(G\times H)$-flow which is simultaneously a minimal $G$-flow and a minimal $H$-flow?
\end{que}

\vspace{5 mm}
\small

\noindent
Institut de Math\'ematiques de Jussieu--PRG

\noindent
Universit\'e Paris Diderot, case 7012 

\noindent
8 place Aur\'elie Nemours 

\noindent
75205 Paris cedex 13, France.
\vspace{3 mm}

\noindent
Email address: \url{andrew.zucker@imj-prg.fr}

\end{document}